\documentclass{cont}

% ...

\usepackage{amsfonts}
\usepackage{amsmath}
\usepackage{bbm}
\usepackage{mathrsfs} 
\usepackage{graphicx}
\usepackage{bm}

% ...

\newcommand{\fourier}[1]{\mathscr{F}\left( #1 \right)}

\equalenv{lem}{lemm}
\equalenv{example}{exam}
\equalenv{twier}{theo}
\equalenv{uwaga}{rema}
\newcommand{\eps}{\varepsilon}

\title
[On the continuity of Fourier multipliers]
{On the continuity of Fourier multipliers on the homogeneous Sobolev
spaces $\bm{\dot{W}^1_1(R^d)}$}

\author{\firstname{Krystian}  \lastname{Kazaniecki}}

\address{University of Warsaw\\ 
Institute of Mathematics\\
ul. Banacha 2\\
02-097 Warszawa,(Poland)}

\email{Krystian.Kazaniecki@mimuw.edu.pl}

\thanks{The research of the first author has been supported by the Polish Ministry
of Science grant no. N N201 397737 (years 2009-2012)}
\author{\firstname{Micha\l}  \lastname{Wojciechowski}}

\address{Polish Academy of Sciences\\ 
Institute of Mathematics\\
ul. \'Sniadeckich 8\\
00-956 Warszawa, (Poland)}

\email{M.Wojciechowski@impan.pl}

\thanks{The research of the second author has been supported by the NCN grant no. N N201 607840}

  

\begin{document}
\begin{abstract}
  In this paper we proof that every Fourier multiplier on homogeneous Sobolev space $\dot{W}_1^1(\mathbb{R}^d)$ is a continuous function. This theorem is generalization of A. Bonami and S. Poornima result for Fourier multipliers, which are homogeneous functions of degree zero. 
\end{abstract}


\maketitle
\section{Introduction}
We consider the invariant operators on the homogeneous Sobolev spaces on $\mathbb R^d$ given by Fourier multipliers. The Sobolev space $\dot{W}^1_1(\mathbb R^d)$ consists of
those functions on $\mathbb R^d$ whose distributional derivatives of order one are
integrable. The pseudonorm, given by $\| \nabla f\|_1$, is a norm on the quotient by
constant functions (cf. \cite{MR879706}). A measurable function $m:\mathbb{R}^d\rightarrow\mathbb{R}$ is called a (Fourier) multiplier if the operator given by the formula
$T_mf=\mathscr{F}^{-1}(m\cdot \fourier{f})$ is bounded.
Fourier transforms of a bounded measures are examples of multipliers. Indeed,
the convolution with a bounded measure is a bounded operator on every translationaly invariant space where shifts operators are continuous, in particular on the homogeneous Sobolev space.\\

However, in this case the class of Fourier multipliers is wider than the class of Fourier transforms of measures (Proposition 2.2 in \cite{MR729352}). One of the most important questions about the invariant subspace of $L^1$
is how singular bounded operators acting on it may be. The class of invariant singular operators, which plays the most important role in analysis, are the Calderon - Zygmund operators which are given (in the invariant case) by the multipliers that are noncontinuous at 0. Therefore, the question of the continuity of a multiplier arises quite naturally in the theory.\\

The simplest case of noncontinuous multipliers was settled by A. Bonami and S. Poornima who proved that the only multipliers which are homogeneous functions of degree 0 are the constant functions. In their beautiful proof they use very delicate result by Ornstein (cf. \cite{MR0149331} ) on the non-majorization of a partial derivative by the other derivatives of the same order.
While the class of homogeneous multipliers, containing e.g. Riesz transforms, is the most important one, the question of the continuity of general multipliers remained open. The aim of this paper is to fill the gap. We prove that any multiplier acting on the homogeneous Sobolev space with integral norm is a continuous function.\\

Our proof uses three main ingredients. The first one is the Bonami - Poornima result.
The second is the Riesz product techniques which allows us to make the crucial estimates
on the torus group which would be sufficient for our purpose, provided we are able to transfer the problem from $\mathbb R^d$ to $\mathbb T^d$. This transference in the case of multipliers on $L^1$ space is the subject of the theorem of deLeeuw. However, in the case of multipliers on the homogeneous Sobolev space no equivalent of the deLeeuw transference theorem is known. We are able to overcome this difficulty
due to the special form of functions on which the multiplier reaches its norm.
The question of general deLeeuw type theorem for the homogeneous Sobolev spaces
remains open and we believe that this paper will provide a motivation for futher research in this direction.\\
For a formal statement of the main theorem we need some auxiliary definitions and notations. \small
\begin{itemize}
\item[$\cdot$] $L^p(\mathbb{R}^d)$ - space of Lebesgue $p$-integrable functions on $\mathbb{R}^d$
\item[$\cdot$] $\mathscr{D}(\mathbb{R}^d)$ - space of $C^{\infty}(\mathbb{R}^d)$ functions with compact support on $\mathbb{R}^d$.
\item[$\cdot$] $\mathscr{D}'(\mathbb{R}^d)$ - space of distributions on $\mathbb{R}^d$.  
\item[$\cdot$] $\mathscr{S}(\mathbb{R}^d)$ - Schwartz function space on $\mathbb{R}^d$. 
\item[$\cdot$] $\mathscr{S}'(\mathbb{R}^d)$ - space of tempered distributions on $\mathbb{R}^d$. 
\item[$\cdot$] $C_b(\mathbb{R}^d)$ - space of bounded continuous functions on $\mathbb{R}^d$. 
\item[$\cdot$] $\bm{G}$ - locally compact topological group.
\item[$\cdot$] $\bm{M}(\bm{G})$ - space of regular, bounded borel measures on $\bm{G}$.
\item[$\cdot$] $\fourier{\cdot}$ - Fourier transform of tempered distributions.
\item[$\cdot$] $\mathscr{F}^{-1}(\cdot)$ - inverse Fourier transform of tempered distributions. 
\end{itemize}
\normalsize
One can find more details on the function spaces mentioned above in \cite{MR1157815}. We define the Fourier transform as in \cite{MR0304972}.As
usual, $C$ will denote a generic constant, whose value can change from line to line.\\
We write $W^{k}_{p}(\mathbb{R}^d)$ for the Sobolev space on $\mathbb{R}^{d}$, given by
\begin{equation*}
W^{k}_{p}(\mathbb{R}^d)=\left\{f\in L^p(\mathbb{R}^d)\, :\, D^{\bm{\alpha}}f\in L^p(\mathbb{R}^d)\mbox{ for } |\bm{\alpha}|\leq k\right\}
\end{equation*}
with the norm
\begin{equation*}
\|f\|_{W^{p}_{k}(\mathbb{R}^d)}=\sum_{0\leq|\bm{\alpha}|\leq k}\|D^{\bm{\alpha}}f\|_{L^p(\mathbb{R}^d)}
\end{equation*}
where $\bm{\alpha}$ is a multi-index and $D^{\bm{\alpha}}$ is the corresponding distributional derivative
and $k\in\mathbb N^{+}$.
Analogously we write $\dot{W}^{k}_{p}(\mathbb{R}^d)$ for the homogeneous Sobolev space on $\mathbb{R}^{d}$, given by
\begin{equation*}
\dot{W}^{k}_{p}(\mathbb{R}^d)=\left\{f\in \mathscr{D}'(\mathbb{R}^d)\, :\, D^{\bm{\alpha}}f\in L^p(\mathbb{R}^d)\mbox{ for } |\bm{\alpha}|= k\right\}
\end{equation*}
with the seminorm
\begin{equation*}
\|f\|_{\dot{W}^{p}_{k}(\mathbb{R}^d)}=\sum_{|\bm{\alpha}|=k}\|D^{\bm{\alpha}}f\|_{L^p(\mathbb{R}^d)}
\end{equation*}
where $\bm{\alpha}$, $D^{\bm{\alpha}}$ and $k$ are the same as above. The homogeneous Sobolev spaces are the special cases of Beppo-Levy spaces which are discussed in \cite{MR0074787}. In the following part of the paper we will use the symbol $\dot{W}^{k}_{p}(\mathbb{R}^d)$ to denote quotient space $\dot{W}^{k}_{p}(\mathbb{R}^d)\slash \mathscr{P}^{k}$, where $\mathscr{P}^{k}$ stands for the space of polynomials of the degree strictly less then $k$. The space $\dot{W}^{k}_{p}(\mathbb{R}^d)\slash \mathscr{P}^{k}$ with quotient norm is a Banach space.
 We say that the function $m\in L^{\infty}(\mathbb{R}^d)$ is a Fourier multiplier on $X$, where $X$ is either the Lebesgue space, the Sobolev space or the homogeneous Sobolev space $\dot{W}_1^1(\mathbb{R}^d)$, if there exists a bounded operator $T : X\rightarrow X$ such that 
\begin{equation*}
\fourier{Tf}=m\fourier{f} \qquad \forall\,f\in\mathscr{S}(\mathbb{R}^d)
\end{equation*}   
We use the symbol $\mathscr{M}(X,X)$ to denote the space of Fourier multipliers on $X$.\\
Now we can state the main result of this paper \\
\begin{twier}\label{glownywynik}
If  $d\geq 2$ and $m(\cdot)\in\mathscr{M}(\dot{W}^1_1(\mathbb{R}^d),\dot{W}^1_1(\mathbb{R}^d))$ then $m(\cdot)\in C_b(\mathbb{R}^d)$.
\end{twier}
In the proof of the main theorem we will use the following theorem of A. Bonami and S. Poornima on the homogeneous Fourier multipliers on $\dot{W}^1_1(\mathbb{R}^d)$. 
\begin{twier}[A. Bonami, S. Poornima]\label{bonami}
Let $\Omega$ be a continuous function on $\mathbb{R}^d\backslash\{0\}$, homogeneous of degree zero i.e.
\begin{equation*}
\Omega(\eps\bm{x})=\Omega(\bm{x})\qquad \forall\,\bm{x}\in\mathbb{R}^d.
\end{equation*}
Then
\begin{equation*}
\Omega\in\mathscr{M}(\dot{W}^1_1(\mathbb{R}^d), \dot{W}^1_1(\mathbb{R}^d)) \Leftrightarrow \Omega\equiv K\in\mathbb{C}
\end{equation*}  
\end{twier}
For the proof see \cite{MR879706}. 
We will also use the classical theorem on pointwise convergence of multipliers on $L^1(\mathbb{R}^d)$ and deLeeuw theorems.
\begin{twier}\label{slaba*miary}
Let $\{\mu_\alpha\}_{\alpha\in A}$ be net of measures in $\bm{M}(\bm{G})$ such that $\|\mu_{\alpha}\|_{\bm{M}(\bm{G})}< M$ for all $\alpha\in A$ and $lim_{\alpha} \fourier{\mu_{\alpha}}(\chi)=\phi(\chi)$ exist for all $\chi\in X$. Suppose that $\phi$ is continuous. Then the limit function $\phi$ has the form $\fourier{\mu}$ for certain  $\mu\in\bm{M}(\bm{G})$ and $\|\mu\|_{\bm{M}(\bm{G})}$ does not exceed $M$.
\end{twier}
Proof of this theorem is in \cite{MR0262773} as Corollary 33.21.

\begin{twier}[deLeeuw]\label{deleeuwzplaszczyznydotorusa}
Let $1\leq p\leq\infty$, $m(\cdot)\in\mathscr{M}(L_p(\mathbb{R}^d),L_p(\mathbb{R}^d))$ and $m(\cdot)$ be continuous in points $\bm{n}\in\mathbb{Z}^d$. Define  
\begin{equation*}
\gamma_{\bm{n}}=m(\bm{n}).
\end{equation*}
Then $\{\gamma_{\bm{n}}\}\in\mathscr{M}(L_p(\mathbb{T}^d),L_p(\mathbb{T}^d))$ and the following inequality holds
\begin{equation*}
\|T_{\gamma}\|\leq \|T_{m}\|
\end{equation*}
\end{twier}
\begin{twier}[deLeeuw]\label{deleeuwztorusadoplaszczyzny}
Let $m(\cdot)$ be continuous function on $\mathbb{R}^d$. Define for every $\eps>0$  
\begin{equation*}
\gamma(\eps)_{\bm{n}}:=m(\eps\bm{n}) \qquad\forall\,\bm{n}\in\mathbb{Z}^d
\end{equation*}
If
\begin{equation*}
\{\gamma(\eps)_{\bm{n}}\}_{\bm{n}\in\mathbb{Z}^d}\in\mathscr{M}(L_p(\mathbb{T}^d),L_p(\mathbb{T}^d))
\end{equation*} 
and
\begin{equation*}
\|T_{\gamma(\eps)}\|\leq C\qquad\forall\;\eps>0
\end{equation*}
Then $m(\cdot)\in\mathscr{M}(L_p(\mathbb{R}^d),L_p(\mathbb{R}^d))$ and
\begin{equation*}
\|T_m\|\leq \sup_{\eps>0}\|T_{\gamma(\eps)}\|
\end{equation*}
\end{twier}
Proof of these two theorems are in \cite{MR0304972}.
\begin{uwaga}
In Theorem \ref{deleeuwztorusadoplaszczyzny} it is sufficient to take convergent to zero a sequence $\{\eps_j\}$ instead of every $\eps>0$.
\end{uwaga}

In the next section we prove the main result. To focus the attention on the main line of the proof, some technical lemmas are formulated in that section without proofs. 
For the reader's convenience proofs of the technical lemmas are given in the last section.

\section{Proof of main theorem}
In this section we prove the main result of this paper (Theorem \ref{glownywynik}). It is obvious that $m(\cdot)\in\mathscr{M}(\dot{W}^1_1(\mathbb{R}^d),\dot{W}^1_1(\mathbb{R}^d))$  is continuous on $\mathbb{R}^d\backslash\{0\}$. Then it is enough to show that 
 $\lim_{\bm{x}\rightarrow 0} m(\bm{x})$ exists.\\ Prior to proof of the theorem we need one more definition.
Let $f: \mathbb{R}^d\to\mathbb{R}$. We say that $f$ has \textit{almost radial limits} at 0 iff the following condition (*) holds:
\begin{description}
\item[(*)] whenever two sequences $\{t_k\bm{v}^k\}$, $\{s_k\bm{w}^k\}$ ($t_k\in\mathbb{R}$, $s_k\in\mathbb{R}$, $\bm{v}^k\in\mathbb{S}^{d-1}$,$\bm{w}^k\in\mathbb{S}^{d-1}$)
satisfy
\begin{equation*}
\begin{split}
\lim_{k\rightarrow \infty }f(t_k\bm{v}^k)=a&\neq b=\lim_{k\rightarrow\infty}f(s_k\bm{w}^k),\\
\lim_{k\rightarrow \infty } t_k&=\lim_{k\rightarrow \infty } s_k=0,\\
\end{split}
\end{equation*}
then $\lim_{k\rightarrow \infty }|\bm{v}^k- \bm{w}^k| > 0$.\\
\end{description}
\begin{proof}{ Theorem \ref{glownywynik}}

Since $m$ is bounded, there are three possibilities:
\begin{enumerate}
\item[I] $m(\cdot)$ satisfies condition (*).
\item[II] Condition (*) does not hold. Then there exists a sequence $\{\bm{a}^{n}\}_{n\in \mathbb{N}}\subset\mathbb{R}^d$, $\bm{a}^{n}\to 0$,  a vector $v\in \mathbb{S}^1$, two different scalars $a$ and $b$ such that 
\[\lim_{n\rightarrow \infty} \frac{\bm{a}^{n}}{|\bm{a}^{n}|} = v\]
and one of the following is satisfied\\
\begin{enumerate}
\item Symmetric case.
\begin{equation}\label{ciagizcaseIIa}
\begin{aligned}
\lim_{n\rightarrow \infty} &m(\bm{a}^{2n} )&=\lim_{n\rightarrow \infty} &m(-\bm{a}^{2n} )&=a\\ \lim_{n\rightarrow \infty} &m(\bm{a}^{2n+1} )&=\lim_{n\rightarrow \infty} &m(-\bm{a}^{2n+1})&=b.
\end{aligned}
\end{equation}
\item Asymmetric case.
\begin{equation*}
\begin{aligned}
&\lim_{n\rightarrow \infty} m(\bm{a}^n )&=a\\& \lim_{n\rightarrow \infty}m(-\bm{a}^n )&=b.
\end{aligned}
\end{equation*}
\end{enumerate}
\end{enumerate}
\subsection{Proof in the case I}
To prove the continuity in this case we need the following lemma on the pointwise convergence of multipliers. 
\begin{lem}\label{punktowazbieznoscmnoznikow}
Let $\{m_k(\cdot)\}$ be a sequence of Fourier multipliers on 
 $\dot{W}^1_1(\mathbb{R}^d)$ and assume that the corresponding operators have commonly bounded norms. If $m_k(\cdot)$ converge pointwise to $m(\cdot)$ and $m(\bm{\xi})\xi_j$ are continuous functions for $\mbox{$j\in\{1,2,\ldots,d\}$}$, then   $m(\cdot)\in\mathscr{M}(\dot{W}^1_1(\mathbb{R}^d),\dot{W}^1_1(\mathbb{R}^d))$.
\end{lem}
The prove of this lemma one can find in the Appendix.
In the next lemma we use Theorem \ref{bonami} to show that the multipliers satisfying condition (*) are continuous.

\begin{lem}\label{gwiazdkalemat}
If $d\geq 2$ and $m(\cdot)\in\mathscr{M}(\dot{W}_1^1(\mathbb{R}^d),\dot{W}_1^1(\mathbb{R}^d))$ satisfies condition  (*), then 
$\lim_{\bm{\xi}\rightarrow 0} m(\bm{\xi})$ exists and is finite.
\end{lem}
\begin{proof}
Note first that $m$ has the radial limit at 0 (we apply (*) to fixed $\bm{v}=\bm{v}^k=\bm{w}^k$). 
Hence the formula
\begin{equation*}
\Omega(\bm{\xi})=\lim_{n\rightarrow \infty} m(\frac{1}{n}\bm{\xi}).
\end{equation*}
defines a homogeneous function on $\mathbb{R}^d\backslash\{0\}$. One can easily check that due to (*) condition $\Omega$ has to be a continuous function on $\mathbb{R}^d\backslash\{0\}$. 
Since the norm of multipliers from $\mathscr{M}(\dot{W}^1_1(\mathbb{R}^d),\dot{W}^1_1(\mathbb{R}^d))$ is invariant under rescaling, the functions $m(\frac{1}{n}\cdot)$ 
are Fourier multipliers with equal norms.
By Lemma \ref{punktowazbieznoscmnoznikow} their pointwise limit, being bounded and continuous on $\mathbb{R}^d\backslash\{0\}$, is a Fourier multiplier.
 Then Theorem \ref{bonami} implies that $\Omega$ is a constant function which in turn means that all radial limits of $m$ are equal. Since the multiplier $m$ satisfies (*) it follows that $m$ is a continuous function. 
\end{proof}

\subsection{Proof in the case IIa}
From now on we assume that $d=2$. This allows us to
simplify the notation yet not loosing the generality. We can also assume, transforming linearly if necessary, that $a=1$, $b=-1$  and $\bm{v}=(1,0)$.
In the proof we will use the following lemma to get estimates on the norm of the multiplier $m$.
\begin{lem}(cf. \cite{MR1649869})\label{lematpolostateczny} 
There is $C>0$ such that for every $n\in\mathbb{N}^{+}$ there exists $M=M(n)$ and a sequence $\{\sigma_j\}_{j=1}^{n}\in \{0,1\}^{n}$ such that
\begin{equation}\label{znaki}
\left\|\sum_{j=1}^{s} \sigma_j\cos\left(2\pi \langle\bm{d}^j,\,\bm{\xi}\rangle\right)\prod_{1\leq k< j}\left(1+\cos\left(2\pi\langle \bm{d}^{k},\,\bm{\xi}\rangle\right)\right)\right\|_{L_1(\mathbb{T}^d)}\geq C s
\end{equation}
 whenever $\{\bm{d}^k\}_{k=1}^{s}\subset\mathbb{Z}^d$ satisfies
\begin{equation*}
|\bm{d}^{k+1}|> M |\bm{d}^{k}|.
\end{equation*} 
\end{lem} 
\begin{uwaga}
The precise value of $M_s$ follows from Theorem 5 on page 563 from \cite{MR0240563} which says that
whenever $\sum_{k=1}^s \big(\frac{|d_k|}{|d_{k+1}|}\big)<K$ then the expression 
appearing in the lemma is equivalent to the similar one with functions 
$\bm{\xi}\mapsto\cos(2\pi\langle\bm{d}^j,\,\bm{\xi}\rangle)$ replaced by
cosines of independent Steinhaus variables. And after this replacement the lemma
is just Lemma 1 of \cite{MR1649869}. Similar, but weaker conditions for this equivalence was found by M. Dechamps \cite{MR641858}. 
\end{uwaga}
The improvement by R. Lata{\l}a in \cite{info} allows us to fix $\sigma_j=(-1)^{j}$ in \eqref{znaki}.
In the rest of the paper we put $N= \left(|\frac{\log(M_s)}{\log(2)}|+2\right)$ 

Let us assume that operator $T_m$ corresponding to multiplier $m$ is bounded. For every $s\in \mathbb{N}$ we will construct function $h_s$, such that
\begin{equation*}
 \|T_m h_s\|_{\dot{W}^{1}_{1}(\mathbb{R}^d)}\geq Cs
\end{equation*}
To do this we fix $\eps>0$, which will be determined later and we construct a sequence of balls
$\textbf{B}(\bm{c}^k,r_k)$ and $\textbf{B}(-\bm{c}^k,r_k)$ where $k\in\{1,2,\ldots, s\}$, such that the following conditions hold:
\begin{enumerate}\label{indukcjaIIa}
\item[A.] $|m(\bm{\xi})-(-1)^k| <\eps$ for $\bm{\xi}\in\bm{B}(\bm{c}^{k},r_{k})\cup\bm{B}(-\bm{c}^{k},r_{k})$ for $k=1,2,\ldots, s$,\vspace{0.4em}
\item[B.] $r_{n}\leq 2^{-N}r_{n+1}$ for $n=1,2,\ldots,s-1$,\vspace{0.4em}
\item[C.] $\bm{c}^{n}\in\mathbb{Q}\times\mathbb{Q}$ for $n=1,2,\ldots,s$,\vspace{0.4em}
\item[D.] $|\bm{c}^{n+1}|<2^{-N}r_n$ for $n=1,2,\ldots,s-1$,\vspace{0.4em}
\item[E.] $|c^n_2|\slash |c^n_1|\leq \frac{1}{3^{s+2}s}$ for $n=1,2,\ldots,s$,\vspace{0.4em}
\item[F.] $|\bm{c}^{n}|<2^{-N}|\bm{c}^{n+1}|$ for $n=1,2,\ldots,s-1$,\vspace{0.4em}
\item[G.] $\{0\}\times\mathbb{R}\cup\mathbb{R}\times \{0\}\notin\bm{B}(\bm{c}^{n},r_{n})\cup\bm{B}(-\bm{c}^{n},r_{n})$ for $n=1,2,\ldots,s$,\vspace{0.4em}
\item[H.] $|c^{n}_i|< 2^{-N} |c^{n+1}_{i}|$ for $n=1,2,\ldots,s$ and $i\in\{1,2\}$.\vspace{0.4em}
\item[I.] \small$\forall_{k\in\{1,\ldots,s\}}\bm{B}(\sum_{j=1}^{k}\zeta_j \bm{c}^j,r_s)  \subset \bm{B}(\zeta_k \bm{c}^k,r_k)$ for $\zeta_k\in\{-1,1\}$ and $\zeta_j\in\{-1,0,1\}.$\vspace{0.4em}
\end{enumerate}
\normalfont
We define sequences $\{\bm{c}^k\}$ and $\{r_k\}$ by backward induction. There is no problem with 
$r_n$ because it is chosen always after $\bm{c}^n$ and by B and G it just need to be
sufficiently small. For $\bm{c}^n$ note that the conditions D and F require only that $\bm{c}^n$ is small enough. For condition A it is enough to take $\bm{c}^n$ as an element of
sequence $\{\bm{a}^{2 k+ (n\; mod\; 2)}\}$ (\ref{ciagizcaseIIa}) and if we additionally select sufficiently big $k$, the conditions E and H are satisfied. At the end we adjust our choice to the condition C: since the rationals are dense in $\mathbb{R}$ and all other inequalities are strict, we can do this in such way that inequalities remain valid. 
\\
The condition I follows from B, D and F. Indeed for $\mbox{$k\in\{1,\ldots,s-1\}$}$,$\zeta_j\in\{-1,0,1\}$, $j\in\{1,...,k-1\}$ and $\zeta_k\in\{-1,1\}$ we have
\begin{equation*}
\sum_{j=1}^{k-1} r_j<2^{-N}\sum_{j=2}^{k-1} r_j+2^{-N}r_{k} <\ldots<\left(\sum_{j=1}^{k}2^{-Nj}\right)r_k<\frac{1}{2}r_k.
\end{equation*}
 Hence
\begin{equation}\label{oszacowaniepromieni}
\begin{split}
|\zeta_k \bm{c}^k-\sum_{j=1}^{k}\zeta_j \bm{c}^j|=|\sum_{j=1}^{k}\zeta_j \bm{c}^j|<\sum_{j=1}^{k-1}|\zeta_j \bm{c}^j|<\sum_{j=1}^{k-1}2^{-N}r_j<\frac{1}{2}r_k.
\end{split}
\end{equation} 
By condition B $r_l<\frac{r_k}{4}$ for $k>l$, by (\ref{oszacowaniepromieni}),
\begin{equation*}
\begin{split}
&\bm{B}(\sum_{j=1}^{k}\zeta_j \bm{c}^j,r_1)  \subset \bm{B}(\zeta_k \bm{c}^k,r_k)
\qquad\forall k\in\{1,2,\ldots,s\}, 
\end{split}
\end{equation*}
\\
Norm of operator $T_m$ is invariant under rescaling and condition $C$ holds. Hence for fixed $s$ multiplying $\bm{c}_j$'s by suitable scalar and rescaling multiplier $m$ by the same scalar, we may assume that $\bm{c}^1, \ldots, \bm{c}^s \in \mathbb{Z}^2$ and conditions A-I are still satisfied.
Note that if $\bm{q}\in \mathbb{Z}^2$ has the representation  
\begin{equation}\label{repr}
\bm{q}=\sum_{i=1}^{s} \zeta_j(\bm{q}) \bm{c}^j\quad\mbox{ where }\zeta_i(\bm{\phi})\in\{-1,0,1\}
\end{equation}  
it is unique. For $\bm{q}\in\mathbb{Q}^2$ we denote by $\chi(\bm{q})$ the number of non zero
summands in the representation (\ref{repr}).We define set
\begin{equation}\label{zbioreklambdas}
 \Lambda_s=\{\bm{q} : \bm{q}=\sum_{i=1}^{s} \zeta_j(\bm{q}) \bm{c}^j; \bm{q}\neq 0\;\mbox{ where }\zeta_i(\bm{\phi})\in\{-1,0,1\}\}
\end{equation}
If $\bm{q}$, $\tilde{\bm{q}} \in\Lambda_s$ are two different vectors then
\begin{equation}\label{odleglosccjwiekszaniz1}
|\bm{q}-\tilde{\bm{q}}|\geq \inf|\bm{c}^j|\geq 1.
\end{equation}
\\
We will construct function $h_s$ in such a way that one of it's derivatives acts like a Riesz product. 
Let
\begin{equation*}
g(t)=\max\{1-|t|,0\}^2
\end{equation*}
and
\begin{equation*}
G(\bm{\xi})=g(\xi_1)g(\xi_2) 
\end{equation*}
For given $\theta\in\mathbb{N}^{+}$ we set
\begin{equation}\label{postacfunkcjifns}
H^{\theta}(\xi)=\sum_{\bm{q}\in\Lambda_s} \frac{1}{2^{\chi(\bm{q})}}G\left(2^{\theta}(\bm{\xi}-\bm{q})\right),
\end{equation}

\begin{theorem}{(G. Gaudry, A. Fig{\`a}-Talamanca cf. \cite{MR0276689})}\label{gaudry}\\
Let $\phi$ be a function on $\mathbb{Z}^d$ and $\phi\in\mathscr{M}(L^p(\mathbb{T}^d),L^p(\mathbb{T}^d))$. Let
\begin{equation*}
 W(\phi)(\bm{\xi}) = \sum_{\bm{n}\in\mathbb{Z}^d} G(\bm{n}-\bm{\xi})\phi(\bm{n}),
\end{equation*}
Then the function $W(\phi)\in\mathscr{M}(L^p(\mathbb{R}^d),L^p(\mathbb{R}^d))$, with a norm no greater than the norm of $\phi$.
\end{theorem}

We check at once that for given $\theta\in\mathbb{N}^{+}$
\begin{equation*}
H^{\theta}(\bm{\xi}) = W(\widehat{R_s})(2^{\theta}\bm{\xi}),
\end{equation*}
where $R_s$ is the modified Riesz product:
\begin{equation*}
R_s(\bm{t})=-1 + \Pi_{k=1}^{s}(1+\cos(2\pi \langle \bm{t}\, ,2^{\theta} \bm{c}^k\rangle ) 
\end{equation*}
Since the norm of multipliers from $\mathscr{M}(\dot{W}^1_1(\mathbb{R}^d),\dot{W}^1_1(\mathbb{R}^d))$ is invariant under rescaling, by Theorem \ref{gaudry} we get 
\begin{coro} For every $\theta\in\mathbb{N}^{+}$ the following inequality is satisfied 
\begin{equation*}
\|\mathscr{F}^{-1}(H^{\theta})\|_{L^1(\mathbb{R}^2)}\leq \|R\|_{L^1(\mathbb{T}^2)}\leq 2
\end{equation*}
\end{coro}
Another property of $H^{\theta}$ is the following
\begin{lemm}\label{szacowanko}
 There exists  $\theta=\theta(s)\in{\mathbb{N}^+}$ such that 
 \begin{equation*}
  \left\|\mathscr{F}^{-1}\left(\frac{\xi_2}{\xi_1}H^{\theta}\right)\right\|_{L^1(\mathbb{R}^2)}\leq C
 \end{equation*}
where constant $C$ is independent of $s$.  
\end{lemm}
The proof of this fact one can find in the Appendix. From now one we put $H:= H^{\theta(s)}$.
\begin{uwaga}
We have to remember that homogeneous, non-constant functions are not multipliers on $L^{1}(\mathbb{R}^d)$. The above lemma holds true only due to the special form of $H^{\theta}$, mainly the strong concentration of its support near $x_1$-axis and because of small size of its support.  
\end{uwaga}

Since $H$ is bounded, continuous and has compact support separated from the axis $\{\xi_1=0\}$, the function
$\frac{H}{\xi_1}$ is a tempered distribution. We define a tempered distribution $h$ by the formula
\begin{equation*}
h(\phi)=\frac{H}{x_1}(\mathscr{F}^{-1}{\psi})\qquad \forall \psi\in \mathscr{S}.
\end{equation*}
By standard properties of the Fourier transform acting on tempered distributions, we get
\begin{equation}\label{pochodnesadobre}
\begin{split}
\fourier{\frac{\partial}{\partial x_1}h}&=H\\
\fourier{\frac{\partial}{\partial x_2}h}&=\frac{\xi_2}{\xi_1}H.
\end{split}
\end{equation}
Since we already proved that both $H$ and $\frac{\xi_2}{\xi_1}H$ are the Fourier
transforms of $L^1$ functions, then (\ref{pochodnesadobre}) means that
$h\in \dot{W}^1_1(\mathbb{R}^d)$ with the norm bounded by a constant independent of $s$.
Now we estimate the norm of $T_m h$ from below. We have
\begin{equation}\label{normaT}
\begin{split}
\|T_m h\|_{\dot{W}^1_1(\mathbb{R}^2)}&=\|\frac{\partial}{\partial x_1}T_{m} h\|_{L_1(\mathbb{R}^2)}.
\end{split}
\end{equation}
Since $T_{m}:\dot{W}^1_1(\mathbb{R}^2)\rightarrow\dot{W}^1_1(\mathbb{R}^2)$, obviously $\frac{\partial}{\partial x_1}T_{m} h\in L_1(\mathbb{R}^2)$. Hence the operator
$Q$ defined by
\begin{equation*}
 Q g=\frac{\partial}{\partial x_1}T_{m} h*g
\end{equation*}
acts on $L_1(\mathbb{R}^2)$ and
\begin{equation}\label{qnl1norma}
\|Q\|=\|\frac{d}{dx}T_{m} h\|_{L_1(\mathbb{R}^2)}.
\end{equation}
By (\ref{pochodnesadobre}),
\begin{equation*}
\fourier{Qg}(\xi)=m(\bm{\xi})H(\bm{\xi})\fourier{g}(\bm{\xi}).
\end{equation*}
We define a function $P$ by the formula
\begin{equation}\label{definicjaPs}
P(\xi)=\sum_{\bm{p}\in\mathbb{Z}^2} m(\bm{p})H(\bm{p}) e^{2\pi i \langle\bm{p},\bm{\xi}\rangle}
\end{equation}
Since $H(\bm{p})$ takes non zero values only for $\bm{p}\in \Lambda_s$ and $\Lambda_s$ is finite, the function $P$ is a polynomial.
By Theorem \ref{deleeuwzplaszczyznydotorusa} we get
\begin{equation}\label{qnnorma}
\|Q\|\geq \|P\|_{L_1(\mathbb{T}^2)}.
\end{equation}
We put
\begin{equation*}
a(\bm{p})=\begin{cases}(-1)^k
H(\bm{p})\;\;\mbox{when } \bm{p}\in\Lambda_s\mbox{ and }\bm{p}\in \bm{B}(c^{k},r_{k})\cup\bm{B}(-c^{k},r_{k}),\\
0\qquad\qquad\;\;\;\mbox{otherwise}.
\end{cases}
\end{equation*}
Since $\Lambda_s$ is a finite set, the function
\begin{equation*}
Z(\bm{\xi})=\sum_{\bm{p}\in\mathbb{N}^2}  a(\bm{p}) e^{2\pi i \langle\bm{p},\bm{\xi}\rangle}
\end{equation*}
is a polynomial.
By the triangle inequality,
\begin{equation}\label{GRnorma}
\|P\|_{L_1(\mathbb{T}^2)}\geq\|Z\|_{L_1(\mathbb{T}^2)}-\|P-Z\|_{L_1(\mathbb{T}^2)}.
\end{equation}
By the conditions I and  A, any coefficient of $Z$ differs by at most $\eps$ from the
corresponding coefficient of $P$. Since both polynomials have no more then $3^s$
non-zero coefficients, we get
\begin{equation}\label{niewielebrakujedoszczescia}
\|Z-P\|_{L_1(\mathbb{T}^2)}\leq \eps 3^{s}.
\end{equation}
It is easy to verify that
\begin{equation*}
Z(\bm{\xi})=\sum_{j=1}^{s} (-1)^j \cos\left(2\pi \langle\bm{c}^{j},\,\bm{\xi}\rangle\right)\prod_{1\leq k< j}\left(1+\cos\left(2\pi\langle\bm{c}^{k},\,\bm{\xi}\rangle\right)\right).
\end{equation*}
By the condition F and Lemma \ref{lematpolostateczny},
\begin{equation*}
\|Z\|_{L_1(\mathbb{T}^2)}\geq Cs.
\end{equation*}
Combining now successively (\ref{normaT}), (\ref{qnl1norma}),
(\ref{qnnorma}), (\ref{GRnorma}) and  (\ref{niewielebrakujedoszczescia}),
we get
\begin{equation*}
\|T_m h\|_{\dot{W}^1_1(\mathbb{R}^2)}\geq Cs - \eps 3^{s},
\end{equation*}
Setting $\eps= C 3^{-s-1} s$
\begin{equation*}
\|T_m h\|_{\dot{W}^1_1(\mathbb{R}^2)}\geq Cs
\end{equation*}
which by the uniform boundedness of $\|h\|_{\dot{W}^1_1(\mathbb{R}^2)}$ proves that $T$ is unbounded.

\subsection{Proof in case IIb}
The proof in this case is very similar to case IIa. The only difference is that due to lack of symmetry we have to replace Lemma \ref{lematpolostateczny} by its asymmetric counterpart. We will use the following result from \cite{MR1649869}.
\begin{lem}\label{lematostateczny}
There exist $C>0$ such that for every $n\in\mathbb{N}^{+}$ there exists $M=M(n)$ such that for any sequence $\{\bm{d}^k\}_{k=1}^{n}\subset\mathbb{Z}^d$, which satisfies
\begin{equation*}
|\bm{d}^{k+1}|> M |\bm{d}^{k}|, 
\end{equation*} 
following inequality holds
\begin{equation*}
\|\sum_{j=1}^{n} e^{2\pi i \langle\bm{d}^j,\bm{\xi}\rangle}\prod_{1\leq k< j}\left(1+\cos\left(2\pi \bm{d}^k\xi\right)\right)\|_{L_1(\mathbb{T}^r)}\geq Cn.
\end{equation*} 
\end{lem}
For fixed $\eps>0$ we construct the sequence of balls $\bm{B}(\bm{c}^n,r_n)$ and $\bm{B}(-\bm{c}^n,r_n)$ satisfying conditions B-I and 
\begin{itemize}\label{indukcjaIIb}
\item[$A^{'}$.] $|m(\xi)-1| <\eps$ for $\bm{B}(\bm{c}^n,r_{n})$ and $|m(\xi)| <\eps$ for $\xi\in\bm{B}(-\bm{c}^n,r_{n})$ and $n=1,2,\ldots,s$,\vspace{0.4em}
\end{itemize}
The inductive construction is similar to that in the case IIa.
Then, analogously an as in the case IIa, we define $\theta(s)$, $h$, and we get
\begin{equation*}
\|h\|_{\dot{W}^{1}_{1}(\mathbb{R}^2)}\leq C,
\end{equation*}
Analogously as in the case IIa we define polynomial $P$ by the formula (\ref{definicjaPs}) and due to the similar reasons 
 \begin{equation*}
\|T_m h\|_{\dot{W}^{1}_{1}(\mathbb{R}^2)}\geq \|P\|_{L_1(\mathbb{T}^2)}.
\end{equation*}
Then we put
\begin{equation*}
a(\bm{p})=\left\{\begin{array}{cl}
H(\bm{p})&\mbox{ when } \bm{p}\in\Lambda_s\mbox{ and }\bm{p}\in \bm{B}(\bm{c}^{k},r_{k}),\\
0&\mbox{ otherwise },
\end{array}\right.
\end{equation*}
where $k\in\{1,2,\ldots,s\}$. The function $a(\cdot)$ differs from its analog in the case IIa.
We define a polynomial $Z$ by
\begin{equation*}
Z(\bm{\xi})=\sum_{\bm{p}\in\mathbb{Z}^2}  a(\bm{p}) e^{2\pi i \langle\bm{p},\bm{\xi}\rangle}.
\end{equation*}  
It is easy to check that
\begin{equation*}
Z(\xi)=\sum_{j=1}^{2n} e^{2\pi i \langle\bm{c}^{j},\bm{\xi}\rangle}\prod_{ 1\leq k<j}\left(1+\cos\left(2\pi \langle\bm{c}^{k},\,\bm{\xi}\rangle\right)\right),
\end{equation*}
and similar reasoning as in the case IIa (\ref{niewielebrakujedoszczescia}) gives 
\begin{equation*}
\|P\|_{L_1(\mathbb{T}^2)}\geq\|Z\|_{L_1(\mathbb{T}^2)}-\eps 3^{s}.
\end{equation*} 
By Lemma \ref{lematostateczny},
\begin{equation*}
\|Z\|_{L_1(\mathbb{T}^2)}\geq Cs.
\end{equation*}
Hence
\begin{equation*}
\|T_m h\|_{\dot{W}^1_1(\mathbb{R}^2)}\geq Cs - \eps 3^{s},
\end{equation*}
and setting $\eps= C 3^{-s-1} s$ we get
\begin{equation*}
\|T_m h\|_{\dot{W}^1_1(\mathbb{R}^2)}\geq Cs
\end{equation*}	
which by uniform boundedness of $\|h\|_{\dot{W}^1_1(\mathbb{R}^2)}$ proves
that $T$ is unbounded
\end{proof}
\section{Appendix}
\subsection{Proof of lemma \ref{punktowazbieznoscmnoznikow}}

First of all we estimate the supremum norm of the multiplier by the norm of the corresponding operator.
\begin{lem}\label{ograniczeniepozazerem}
Let $m(\bm{\xi})$ be in $\mathscr{M}(\dot{W}^1_1(\mathbb{R}^d),\dot{W}^1_1(\mathbb{R}^d))$, then
\begin{equation*}
\|m\|_{C(\mathbb{R}^d\backslash\{0\})}\leq \|T_m\|.
\end{equation*}
\end{lem}
\begin{proof}
Using the invariance of the class of multipliers by dilation and rotation, it is sufficient to prove that $|m(\xi)| \leq \|T_m\|$ for $\xi$ a fixed point, say $\xi = (1, 0, \ldots , 0)$. Let $f\geq 0$ and $f\in\mathscr{S}$. We will test the operator $T_m$ on the family of functions
\begin{equation}
 h_{\lambda} (\bm{x}) = \lambda e^{ 2iπ\langle \bm{x},\xi\rangle} f (\lambda \bm{x}).
\end{equation}
On one hand the value at $\bm{\xi}$ of the Fourier transform of $\partial_1 (T_m h_{\lambda})$ is equal to
$2\pi m(\bm{\xi}) \mathscr{F}(f) (0)$. On the other hand we compute the gradient and find that
\begin{equation}
| \nabla h_{\lambda} (\bm{x})| \leq 2\pi \lambda |f (\lambda \bm{x})| + \lambda^2 |\nabla f (\lambda \bm{x})|.
\end{equation}
Hence
\begin{equation}
 |2 \pi m(\bm{\xi}) \mathscr{F}(f) (0)| \leq \|T_m\|\| \nabla h_{\lambda}\|_{1}\leq\|T_m\|\left(  2\pi\mathscr{F}(f)(0) + \lambda\|\nabla f\|_{L^1(\mathbb{R}^d)}\right)
\end{equation}
The conclusion follows at once, by having $\lambda$ tend to zero.

\end{proof}

Now we can prove Lemma \ref{punktowazbieznoscmnoznikow}. We derive it from the corresponding result for measures (Bochner's theorem).
\begin{proof}{of the Lemma \ref{punktowazbieznoscmnoznikow}}.\\
By Lemma \ref{ograniczeniepozazerem} 
\begin{equation*}
\|m\|_{C(\mathbb{R}^d\backslash\{0\})}\leq \tilde{C}.
\end{equation*}
Hence $m(\cdot)\in L^{\infty}(\mathbb{R}^d)$. 
Clearly for $j\in\{1,2,\ldots, d\}$ and $f\in\mathscr{S}(\mathbb{R}^d)$,
\begin{equation*}
\|\frac{\partial}{\partial x_j} T_{m_k}f\|_{L_1(\mathbb{R}^d)}\leq C \|f\|_{\dot{W}^1_1(\mathbb{R}^d)}.
\end{equation*}

For fixed $f$, by *-weak compactness of unit ball in $\bm{M}(\mathbb{R}^d)$ there exists a sequence $m_{k_j}(\cdot)$ such that for $j\in\{1,2,\ldots, d\}$ there exist measures $\mu_j$ for which
\begin{equation*}
\frac{\partial}{\partial x_j} T_{m_{k_l}}f \stackrel{w*}\longrightarrow  \mu_j.
\end{equation*}
For  $f\in\mathscr{S}(\mathbb{R}^d)$ we have
\begin{equation*}
\mathscr{F}(\frac{\partial}{\partial x_j} T_{m_{k_l}}f)(\bm{\xi})=m_{k_l}(\bm{\xi})\xi_j\fourier{f}(\bm{\xi}).
\end{equation*} 
By the assumptions of the Lemma \ref{punktowazbieznoscmnoznikow}, the pointwise limits of Fourier transforms of these functions
exist and are continuous. By Theorem
\ref{slaba*miary}, for $j\in\{1,2,\ldots, d\}$,
\begin{equation*}
\begin{split}
\fourier{\mu_j}(\bm{\xi})&= m(\bm{\xi})\xi_j\fourier{f}(\bm{\xi}),\\
\|\mu_j\|_{\bm{M}(\mathbb{R}^d)}&\leq C \|f\|_{\dot{W}^1_1(\mathbb{R}^d)},
\end{split}
\end{equation*}
Since $f\in\mathscr{S}(\mathbb{R}^d)$ and $m(\cdot)\in L_{\infty}(\mathbb{R}^d)$, it follows that $m(\cdot)\fourier{f}\in L_{2}(\mathbb{R}^d)$. Then  $\mathscr{F}^{-1}(m(\cdot)\fourier{f})\in L_{2}(\mathbb{R}^d)$. Repeating this for $\xi_j f$ we get that $\mu_j$ is a function. Therefore $\mu_j\in L_{1}(\mathbb{R}^d)$ and
\begin{equation*}
\|\mu_j\|_{L_{1}(\mathbb{R}^d)}=\|\mu_j\|_{\bm{M}(\mathbb{R}^d)}\leq C \|f\|_{\dot{W}^1_1(\mathbb{R}^d)}.
\end{equation*} 
Since the Fourier transform is bijective on tempered distributions,
\begin{equation*}
\frac{\partial}{\partial x_j}(\mathscr{F}^{-1}(m(\cdot)\fourier{f}))=\mu_j.
\end{equation*}
Hence it is easy to check that  
\begin{equation*}
\|T_mf\|_{\dot{W}^1_1(\mathbb{R}^d)}\leq C\|f\|_{\dot{W}^1_1(\mathbb{R}^d)} \quad \mbox{for }\,f\in\mathscr{S}(\mathbb{R}^d).
\end{equation*}
and $T_m$ could be uniquely extended to bounded operator on  $\dot{W}^1_1(\mathbb{R}^d)$.  
\end{proof}

\subsection{Proof of lemma \ref{szacowanko}}
We begin with two lemmas.
We study the operator given by sufficiently smooth multiplier acting on a subspace of $L^1$ functions with compactly supported Fourier transform. Let $k$ be the smallest even number greater then $\lceil\frac{d}{2}\rceil$, $d\geq 2$. We fix function $\eta \in C^{\infty}_0$ supported in ball of radius $1$.  
\begin{lem}\label{jakpozbycsieulamkow}
Let $0<\eps \leq r <1$  and $f\in C^{k+1}(\bm{B}(0,r))$ with all derivatives of order less or equal to k that
vanish at 0 Then for every following inequality holds
\begin{equation}
\|\mathscr{F}^{-1}(\eta_{\eps}f)\| \leq C(\eta) \eps \sup_{|\bm{x}|\leq 1} \left( \sum_{|\bm{\alpha}|=k+1} |D^{\bm{\alpha}} f (\bm{x})|\right),
\end{equation}
where $\eta_{\eps}(x)=\eta(\eps x)$.
\end{lem}
\begin{proof} We recall that for such $k$ the left hand side is bounded up to a constant by $\|\eta f\|_{W^1_k}$ (cf. \cite{MR0304972}). By the Leibnitz Formula, it is sufficient to prove that all derivatives $D^{\alpha} f$ are dominated by $\sup_{|\bm{x}|\leq 1} \left(\sum_{|\bm{\alpha}|=k+1} |D^{\bm{\alpha}} f (\bm{x}|\right)$ for $|\alpha|\leq k$.
This is a consequence of Taylor's Formula.
\end{proof}

\begin{lem}\label{jakpozbycsieulamkow2}
Let  $0<\eps\leq r\leq 1$ and $f\in C^{k+1}(\bm{B}(0,r))$ then following inequality holds
\begin{equation}\label{nierownoscjakpozbycsieulamkow}
\|\mathscr{F}^{-1}(\eta_{\eps}f)\| \leq C(\eta) \left( |f(0)|+ \eps \sup_{|\bm{x}|\leq 1} \left( \sum_{|\bm{\alpha}|\leq k+1} |D^{\bm{\alpha}} f (\bm{x})|\right)\right)
\end{equation}
\end{lem}
\begin{proof}
Writing $f$ as the sum of a polynomial of degree $k$ and a function satisfying the assumptions of the previous lemma, we see that it is sufficient to consider only polynomials, and, by linearity monomials. For $f(\xi) = (2i\pi\bm{\xi})^{\bm{\alpha}}$, we have 
\begin{equation}
 \|\mathscr{F}^{-1}(\eta_{\eps} f )(\bm{x})\|_{L_1}=\|\eps^{d+\alpha} D^{\alpha} \eta(x)\|_{L_1}\leq C \eps^{\alpha}
\end{equation}
Hence inequality \eqref{nierownoscjakpozbycsieulamkow} follows.
\end{proof}

Now we can prove the Lemma \ref{szacowanko}.
\begin{proof}{Lemma \ref{szacowanko}}.\\
By the definition of $H^{\theta}$ we see that its support is contained in the union of disjoint balls of radius $r$ centered in points from $\Lambda_s$. Radius $r$ depends only on the parameter $\theta$, so we can choose it as small as we wish. Let $\eta_{\bm{q}}\in C^{\infty}$ be rescaled and translated copies of the same function $\eta$ with $\operatorname{supp}\eta_{\bm{q}}\subset B(\bm{q},2r)$ and $\eta_{\bm{q}}(\bm{\xi})=1$ ,$\bm{xi}\in B(\bm{q},r)$ for every $q\in\Lambda_s$. The following identity holds
\begin{equation}\label{Hsrozbicie}
\frac{\xi_2}{\xi_1}H^{\theta}(\bm{\xi})=\sum_{\bm{\phi}\in\Lambda_s}\eta_{\bm{q}}(\bm{\xi})\frac{\xi_2}{\xi_1}H^{\theta}(\bm{\xi}).
\end{equation}
By the condition G (page \pageref{indukcjaIIa}) the function $f= \frac{\xi_2}{\xi_1}$ satisfies conditions of the Lemma \ref{jakpozbycsieulamkow2}. Hence for $r$ small enough by the triangle inequality, \eqref{nierownoscjakpozbycsieulamkow} and \eqref{Hsrozbicie} 
\begin{equation*}
\begin{split}
\|\mathscr{F}^{-1}(\eta_{\bm{q}}f H^{\theta})\|_{L_1(\mathbb{R}^2)} &\leq C(\eta) \sum_{\bm{q}\in\Lambda_s} \left( |f(\bm{q})|+ \eps \sup_{|\bm{x}-\bm{q}|\leq 1} \left( \sum_{|\alpha|\leq k+1} |D^\alpha f (\bm{x})|\right)\right)\\&\quad\cdot\|\mathscr{F}^{-1}(H^{\theta})\|_{L_1(\mathbb{R}^2)}.
\end{split}
\end{equation*}
By conditions E and H,
\begin{equation*}
\left|\frac{q_2}{q_1}\right|=\left|\frac{c^k_2+\sum_{j=1}^{k-1} \zeta_j c^j_2}{c^k_1+\sum_{j=1}^{k-1} \zeta_j c^j_1}\right|\leq \frac{k |c^k|}{|c^k_1|-\sum_{j=1}^{k-1} |c^j_1|}\leq \frac{s}{3}\left|\frac{c^k_1}{c^k_2}\right|\leq \frac{1}{2\cdot 3^s}. 
\end{equation*} 
Since $|\Lambda_s|\leq 3^s$ we can choose $\eps>0$ such that 
\begin{equation*}
\begin{split}
\|\mathscr{F}^{-1}(\frac{\xi_2}{\xi_1}&H^{\theta})\|_{L_1(\mathbb{R}^2)} \leq C \|\mathscr{F}^{-1}(H^{\theta})\|_{L_1(\mathbb{R}^2)},
\end{split}
\end{equation*}
where the constant $C$ does not depend on $s$.
\end{proof}

\nocite{*}
\bibliographystyle{cont}
\bibliography{cont}
\end{document}